\documentclass{article}

\usepackage{amsmath,amsthm}
\usepackage[title]{appendix}
\usepackage[german]{babel}
\usepackage{csquotes}
\usepackage{drawmatrix}
\usepackage[inline]{enumitem}
\usepackage[margin=3cm]{geometry}
\usepackage{graphicx}
\usepackage{hyperref}
\usepackage{mathtools}
\usepackage{multirow}
\usepackage{tikz}

\renewcommand{\vec}[1]{{\mathbf{#1}}}

\newcommand{\uv}{\vec{u}}
\newcommand{\vv}{\vec{v}}
\newcommand{\wv}{\vec{w}}

\newcommand{\setbr}[2]{\{{#1}\;|\;{#2}\}}

\newcommand{\N}{\mathbf{N}}
\newcommand{\C}{\mathbf{C}}

\newcommand{\tr}{^\top}
\newcommand{\inv}{^{-1}}
\newcommand{\trace}{\mathrm{tr}}
\newcommand{\linspan}{\mathrm{span}}
\newcommand{\colsp}{\mathcal{C}}
\newcommand{\nullsp}{\mathcal{N}}
\newcommand{\Eig}{\mathrm{Eig}}
\newcommand{\Hau}{\mathrm{Hau}}
\newcommand{\mult}{\mu}

\newcommand{\defas}{\coloneqq}
\newcommand{\asdef}{\eqqcolon}

\newtheorem{theorem}{Satz}
\newtheorem{proposition}[theorem]{Proposition}
\newtheorem{lemma}[theorem]{Lemma}

\theoremstyle{remark}
\newtheorem{example}[theorem]{Beispiel}
\newtheorem{remark}[theorem]{Bemerkung}

\theoremstyle{definition}
\newtheorem{definition}[theorem]{Definition}

\title{Nicht-algebraische Didaktik nicht-diagonalisierbarer Matrizen \\ \smallskip \large Eine angewandte Reise zur Jordanschen Normalform}

\author{Mario Teixeira Parente}
\date{
    \normalsize
    Fakultät für Technik \\
    Hochschule Pforzheim \\
    \href{mailto:mario.parente@hs-pforzheim.de}{\texttt{mario.parente@hs-pforzheim.de}}
}

\begin{document}

\maketitle

\begingroup
\renewcommand{\abstractname}{Abstract}
\begin{abstract}
    This article offers a motivating travel guide towards the Jordan normal form, one of the highlights in courses on linear algebra or advanced mathematics.
    Its itinerary is characterized by a focus on core geometric aspects and the avoidance of algebraic tools.
    In this way, it attempts to encourage academic lecturers from more applied mathematical contexts to devote more time to one of the most exciting structures of linear algebra in their courses, rather than rushing through it, often using just a definition.
\end{abstract}
\endgroup

\begin{abstract}
    Dieser Artikel bietet einen motivierenden Reiseführer zur Jordanschen Normalform, einem der Höhepunkte in Lehrveranstaltungen zu Linearer Algebra oder Höherer Mathematik, an.
    Seine Reiseroute zeichnet sich durch die Besinnung auf geometrische Kernaspekte und den Verzicht auf algebraische Hilfsmittel aus.
    Er versucht so, Hochschuldozenten aus mathematisch angewandteren Kontexten zu ermutigen, einer der spannendsten Strukturen der Linearen Algebra mehr Zeit in ihren Lehrveranstaltungen einzuräumen, anstatt allzu hastig, und nicht selten nur mittels einer Definition, über sie hinweg zu gehen.
\end{abstract}

\medskip

\begin{center}
    \begin{minipage}{0.95\linewidth}
        \begin{small}
            \textbf{Stichwörter} ---
            Lineare Algebra, Basiswechsel, Trigonalisierung, \textit{A}-Invarianz, Hauptraum, Jordan-Kette
            
            \textbf{MSC 2020} --- 15A18, 15A20, 97H60
        \end{small}
    \end{minipage}
\end{center}

\section{Einleitung}

Das Studium der \enquote{Geometrie} einer Matrix\footnote{Aufgrund des mathematisch angewandten Kontexts verwenden wir in diesem Artikel Matrizen aus~$\C^{n \times n}$ bzgl. der Standardbasis von~$\C^n$ anstatt koordinatenfreie Endomorphismen auf einem endlichdimensionalen Vektorraum über einem algebraisch abgeschlossenen Körper.} $A \in \C^{n \times n}$ findet seinen Höhepunkt in der Jordanschen Normalform von~$A$.
Denn mit ihr steht ein Basiswechsel zur Verfügung, mit dem man~$A$ durch eine \enquote{möglichst diagonale} Matrix darstellen und somit Einsicht in ihr Abbildungsverhalten auf einfache geometrische Figuren wie Kugeln oder Würfel gewinnen kann.
In Anwendungen, die ein z.\,B. mechanisches Phänomen durch ein System linearer Differentialgleichungen (mit konstanten Koeffizienten) modellieren, kann man dadurch beispielsweise dessen unabhängige Variablen \enquote{entkoppeln} und somit das Verständnis des Modells verbessern.

Der Weg zur Jordanschen Normalform ist in einem Mathematik-Studium üblicher\-weise gepflastert mit algebraischen Hilfsmitteln und Begriffen wie charakteristischen Polynomen (definiert mittels Determinanten) oder Minimalpolynomen sowie mit der Untersuchung nilpotenter Matrizen; interessante und notwendige Konzepte, die allerdings allzu leicht von geometrischen Kernaspekten, auf die es uns hier ankommen soll, ablenken.
Mathematische Lehrveranstaltungen in mathematiknahen Studiengängen (wie Informatik, Physik) oder Ingenieur-Studiengängen räumen ihr hingegen oft allgemein wenig Zeit ein und führen sie sogar teilweise, wenn überhaupt, nur mittels einer Definition ein; eine Motivation bleibt dabei aus.

Als Gegenmaßnahme schlägt dieser Artikel eine motivierte Reiseroute zur Jordanschen Normalform vor, die (bis auf den Nachweis der Existenz von Eigenwerten) ohne algebraische Hilfsmittel auskommt und sich über natürliche Etappen dem Reiseziel nähert.
Der Reiseführer legt Wert darauf, dass die Etappenziele der Reise stets klar und motiviert sowie ohne zu steile Anstiege erreichbar sind.
Am Reiseziel angekommen sollen Studierende vor allem in der Lage sein, anschaulich und überzeugend zu erklären, wie Jordan-Blöcke und die Einsen auf ihren Nebendiagonalen entstehen.
Ihre Anschauung soll dabei insbesondere durch Konzentration auf die geometrische Sichtweise und der damit einhergehenden Intuition geschult werden.

Die Zielgruppe des Artikels sind Dozenten mathematischer Grundlagenlehrveran\-staltungen in Ma\-thematik- und mathematiknahen Studiengängen an Hochschulen für angewandte Wissenschaften sowie mathematiknahen und Ingenieur-Studiengängen an Universitäten.
Doch selbst für Lehrende in universitären Mathematik-Studiengängen hält der Text womöglich die ein oder andere Inspiration bereit.

Der Autor betont, dass keine der hier formulierten mathematischen Aussagen unbekannt ist.
Der Mehrwert des Artikels liegt demnach in der Darlegung eines für den Autor nach bestem Wissen neuen didaktischen Erzählstrangs, der den Zugang zur Jordanschen Normalform und ihren geometrischen Hintergründen für Studierende aus mathematisch angewandteren Kontexten erleichtert.

Kürzere Beweise zu mathematischen Aussagen sind im Text untergebracht.
Die restlichen Beweise sind in Anhang~\ref{sec:app_proofs} zu finden.
Darunter sind auch solche, die in Lehrveranstaltungen durchaus übersprungen werden können, ohne dabei die erfolgreiche Vermittlung grundlegender Gedanken und Konzepte zu gefährden.

\section{Reise}

Um unsere Reise motivieren zu können, müssen wir zunächst einen Begriff einführen.

\begin{definition}[Ähnlichkeit]
    Eine Matrix~$A \in \C^{n \times n}$ heißt \textit{ähnlich} zu einer Matrix~$B \in \C^{n \times n}$, falls eine invertierbare Matrix~$S \in \C^{n \times n}$ existiert, sodass $A = S B S\inv$.
\end{definition}

\begin{remark}
    Ähnliche Matrizen unterscheiden sich demnach lediglich durch einen Basiswechsel.
    Sie stellen also, bis auf diesen Basiswechsel, die gleiche lineare Abbildung dar.
\end{remark}

Eine Matrix heißt \textit{diagonalisierbar}, falls sie ähnlich zu einer Diagonalmatrix ist.
Ist $A \in \C^{n \times n}$ diagonalisierbar, existiert also eine invertierbare Matrix $V \in \C^{n \times n}$, sodass
\begin{equation}
    A = V \Lambda V\inv
\end{equation}
für eine Diagonalmatrix $\Lambda \in \C^{n \times n}$.
Es gilt somit insbesondere
\begin{equation}
    AV = V \Lambda \quad \text{bzw.} \quad A\vv_i = \lambda_i \vv_i \quad \text{für $i = 1, \ldots, n$},
\end{equation}
wobei die Vektoren $\vv_i \in \C^n$, d.\,h. die Spalten von~$V$, eine Basis von~$\C^n$ bilden.
Es ist diese Gleichung, die uns die Geometrie einer diagonalisierbaren Matrix verstehen lässt.
In der Tat, für jeden Basisvektor~$\vv_i$ können wir feststellen, ob er durch eine Abbildung von~$A$ (abgesehen von seiner Richtung) lediglich gestaucht ($|\lambda_i| < 1$), gestreckt ($|\lambda_i| > 1$) oder nicht verändert ($|\lambda_i| = 1$) wird.
Somit können wir für \textit{jeden} Vektor $\vv \in \C^n$ nachvollziehen, in welche Richtungen er durch eine Abbildung von~$A$ geometrisch verändert wird.
Mit
\begin{equation}
    \vv = \alpha_1 \vv_1 + \cdots + \alpha_n \vv_n \quad \text{für $\alpha_1, \ldots, \alpha_n \in \C$}
\end{equation}
gilt nämlich
\begin{equation}
    A\vv = \alpha_1 A\vv_1 + \cdots + \alpha_n A\vv_n = \lambda_1 (\alpha_1 \vv_1) + \cdots + \lambda_n (\alpha_n \vv_n).
\end{equation}

Die Motivation unserer Reise ist nun die bedauerliche Feststellung, dass es Matrizen gibt, die nicht diagonalisierbar sind.

\begin{example}
    Die Matrix
    \begin{equation}
        A \defas \begin{pmatrix}
            1 & 1  \\
            0 & 1
        \end{pmatrix}
    \end{equation}
    mit ihrem einzigen Eigenwert~$1$ ist nicht diagonalisierbar.
    Denn wäre~$A$ diagonalisierbar, dann würde eine invertierbare Matrix~$S \in \C^{2 \times 2}$ existieren, sodass
    \begin{equation}
        A = SIS\inv = SS\inv = I \not= A.
    \end{equation}
\end{example}

Um die Geometrie einer nicht-diagonalisierbaren Matrix~$A \in \C^{n \times n}$ dennoch bestmöglich zu verstehen, müssen wir an diesem Ausgangspunkt unserer Reise die entscheidende Frage stellen, ob es eine Basis von~$\C^n$ gibt, mit der~$A$ durch eine \enquote{möglichst diagonale} Matrix, d.\,h. eine Matrix, die abseits der Diagonale möglichst viele Nulleinträge besitzt, dargestellt werden kann.

Unsere Reise soll über vier Etappen zum Ziel, d.\,h. zu einer ausgewogenen Beantwortung der eben gestellten Frage, führen (vgl. Abb.~\ref{fig:journey}):
\begin{enumerate}
    \item Trigonalisierung
    \item Blockdiagonalisierung
    \item Blockweise Trigonalisierung
    \item Jordanisierung
\end{enumerate}

\begin{figure}
    \centering
    \includegraphics[width=0.975\linewidth]{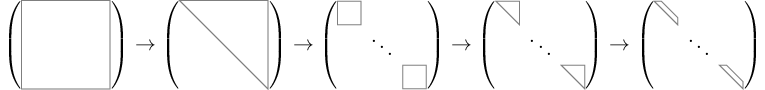}
    \caption{Unsere Reiseroute: Trigonalisierung $\to$ Blockdiagonalisierung $\to$ Blockweise Trigonalisierung $\to$ Jordanisierung.}
    \label{fig:journey}
\end{figure}

\subsection*{Reisevorbereitung}

Bevor wir die Reise antreten, formulieren wir Begriffe und bekannte Aussagen, die wir währenddessen benötigen und deshalb voraussetzen möchten.

\begin{theorem}[Basisergänzungssatz für~$\C^n$]
    \label{thm:basis_compl}
    Sei $\mathcal{M} \subset \C^n$ linear unabhängig.
    Dann existiert eine Basis~$\mathcal{B}$ von~$\C^n$ mit $\mathcal{B} \supseteq \mathcal{M}$.
\end{theorem}

\begin{definition}[Direkte Summe]
    Seien $U,W$ Untervektorräume von~$\C^n$.
    Die Summe
    \begin{equation*}
        U + W \defas \setbr{\uv + \wv}{\uv \in U, \wv \in W}
    \end{equation*}
    heißt \textit{direkt}, falls $U \cap W = \{\vec{0}\}$.
    In Zeichen: $U + W = U \oplus W$.
\end{definition}

\begin{definition}[Spaltenraum, Nullraum]
    Sei $A \in \C^{m \times n}$.
    Die Menge
    \begin{equation}
        \colsp(A) \defas \setbr{A\vv}{\vv \in \C^n} \subseteq \C^m
    \end{equation}
    heißt \textit{Spaltenraum von~$A$}.
    Die Menge
    \begin{equation}
        \nullsp(A) \defas \setbr{\vv}{A\vv = \vec{0}} \subseteq \C^n
    \end{equation}
    heißt \textit{Nullraum von~$A$}.
\end{definition}

\begin{remark}
    Der Spaltenraum und Nullraum einer Matrix sind Untervektorräume von~$\C^n$.
\end{remark}

\begin{theorem}[Rangsatz]
    Sei $A \in \C^{m \times n}$.
    Dann gilt
    \begin{equation}
        \dim \colsp(A) + \dim \nullsp(A\tr) = m
        \quad \text{und} \quad
        \dim \colsp(A\tr) + \dim \nullsp(A) = n.
    \end{equation}
\end{theorem}

\begin{remark}
    Da $\dim \colsp(A\tr) = \dim \colsp(A)$, gilt insbesondere
    \begin{equation}
        \dim \colsp(A) + \dim \nullsp(A) = n.
    \end{equation}
\end{remark}

\begin{definition}[Eigenwert]
    Sei~$A \in \C^{n \times n}$.
    Ein Skalar~$\lambda \in \C$ heißt \textit{Eigenwert von~$A$}, falls ein Vektor~$\vv \in \C^n$ mit $\vv \not= \vec{0}$ existiert, sodass
    \begin{equation}
        \label{eq:eig}
        A\vv = \lambda\vv.
    \end{equation}
    In diesem Fall heißt~$\vv$ ein \textit{Eigenvektor zu~$\lambda$}.
    Die Menge aller Eigenwerte von~$A$, d.\,h.
    \begin{equation}
        \sigma(A) \defas \setbr{\lambda \in \C}{\text{$\lambda$ Eigenwert von~$A$}},
    \end{equation}
    heißt \textit{Spektrum von~$A$}.
    Die Menge
    \begin{equation}
        \Eig_A(\lambda) \defas \nullsp(A - \lambda I)
    \end{equation}
    heißt \textit{Eigenraum zu~$\lambda$}.
\end{definition}

Eine Matrix~$A \in \C^{n \times n}$ ist also genau dann nicht diagonalisierbar, wenn die Eigenräume zu ihren Eigenwerten nicht genügend linear unabhängige Eigenvektoren zur Bildung einer Basis von~$\C^n$ besitzen, d.\,h. falls
\begin{equation}
    \sum_{\lambda \in \sigma(A)} \dim \Eig_A(\lambda) < n.
\end{equation}

\begin{theorem}
    \label{thm:ev_exist}
    Jede Matrix $A \in \C^{n \times n}$ besitzt einen Eigenwert, d.\,h. $\sigma(A) \not= \emptyset$.
\end{theorem}

Ein determinantenfreier Beweis von Satz~\ref{thm:ev_exist}, der als einzige Ausnahme dieses Artikels ein algebraisches Hilfsmittel (Fundamentalsatz der Algebra) verwendet, ist in Anhang~\ref{ssec:app_proof_thm_ev_exist} zu finden.

\begin{proposition}
    \label{prop:eig_simil}
    Ähnliche Matrizen besitzen die gleichen Eigenwerte.
\end{proposition}

\subsection{Trigonalisierung}

Unser erstes Etappenziel ist der Nachweis der Trigonalisierbarkeit einer jeden Matrix $A \in \C^{n \times n}$.
Wir möchten also zeigen, dass~$A$ ähnlich zu einer (oberen) Dreiecksmatrix~$U \in \C^{n \times n}$ ist.
Dafür müssen wir einen Basiswechsel mittels einer invertierbaren Matrix~$V \in \C^{n \times n}$ finden, sodass
\begin{equation}
    \label{eq:schur_decomp}
    A = V U V\inv.
\end{equation}
Damit wären wir dem Gesamtziel unserer Reise, einen Basiswechsel zu einer \enquote{möglichst diagonalen} Matrix zu finden, bereits ein gutes Stück nähergekommen, da unterhalb der Diagonale ausschließlich Nulleinträge entstehen würden.

\begin{theorem}
    Jede Matrix $A \in \C^{n \times n}$ ist trigonalisierbar.
\end{theorem}
\begin{proof}
    Wir beweisen die Aussage per Induktion über~$n$.
    Der Induktionsanfang~$n=1$ ist dabei trivial.
    
    Sei also $A \in \C^{(n+1) \times (n+1)}$.
    Wir wählen zunächst einen Eigenwert $\lambda \in \C$ von~$A$ (Satz~\ref{thm:ev_exist}) und einen zugehörigen Eigenvektor $\vv \not= \vec{0}$, d.\,h. $A\vv = \lambda\vv$.
    Ergänzen wir~$\{\vv\}$ mit Vektoren $\vv_1, \ldots, \vv_n \in \C^{n+1}$ zu einer Basis von~$\C^{n+1}$ (Satz~\ref{thm:basis_compl}) und definieren
    \begin{equation}
        \label{eq:schur_basis_compl}
        V' \defas
        \begin{pmatrix}
            |     &        & |     \\
            \vv_1 & \cdots & \vv_n \\
            |     &        & |
        \end{pmatrix}
        \quad \text{und} \quad
        V \defas
        \begin{pmatrix}
            |   & \multirow{3}{*}{\drawmatrix{$V'$}} \\
            \vv &                                    \\
            |   &                                    \\
        \end{pmatrix},
    \end{equation}
    dann gilt
    \begin{equation}
        AV =
        \begin{pmatrix}
            |          & \multirow{3}{*}{\drawmatrix{$AV'$}} \\
            \lambda\vv &                                     \\
            |          &                                     \\
        \end{pmatrix}
        =
        V
        \begin{pmatrix}
            \lambda & *  \\
            0       & B
        \end{pmatrix}
    \end{equation}
    für eine Matrix $B \in \C^{n \times n}$, also
    \begin{equation}
        A = V
        \begin{pmatrix}
            \lambda & *  \\
            0       & B
        \end{pmatrix}
        V\inv.
    \end{equation}
    Da~$B$ trigonalisierbar ist (Induktionsannahme), gibt es eine obere Dreiecksmatrix $U \in \C^{n \times n}$ und eine invertierbare Matrix $W \in \C^{n \times n}$, sodass $B = W U W\inv$.
    Damit gilt schließlich
    \begin{equation}
        A = V
        \begin{pmatrix}
            \lambda & *          \\
            0       & W U W\inv
        \end{pmatrix}
        V\inv
        =
        \underbrace{
            V
            \begin{pmatrix}
                1 & 0  \\
                0 & W
            \end{pmatrix}
        }_{\asdef \tilde{V}}
        \underbrace{
            \begin{pmatrix}
                \lambda & *  \\
                0       & U
            \end{pmatrix}
        }_{\asdef \tilde{U}}
        \underbrace{
            \begin{pmatrix}
                1 & 0      \\
                0 & W\inv
            \end{pmatrix}
            V\inv
        }_{= \tilde{V}\inv}
        = \tilde{V} \tilde{U} \tilde{V}\inv
    \end{equation}
    mit der oberen Dreiecksmatrix~$\tilde{U} \in \C^{n \times n}$ und der invertierbaren Matrix $\tilde{V} \in \C^{n \times n}$.
\end{proof}

Eine Zerlegung wie in~$\eqref{eq:schur_decomp}$ heißt eine \textit{Schur-Zerlegung von~$A$}; sie ist nicht eindeutig.
Die Matrix~$U$ einer Schur-Zerlegung heißt eine \textit{Schursche Normalform von~$A$}.

\begin{remark}
    \label{rem:trig_eig}
    Da ähnliche Matrizen die gleichen Eigenwerte besitzen (Prop.~\ref{prop:eig_simil}) und die Eigenwerte einer Dreiecksmatrix auf ihrer Diagonale stehen, bilden die Diagonaleinträge einer Schurschen Normalform von~$A$ das Spektrum von~$A$.
\end{remark}

Das Etappenziel ist hiermit erreicht.
Zur Motivation für die Fortsetzung unserer Reise stellen wir fest, dass die Basis in~\eqref{eq:schur_basis_compl} zwar mit passenden, aber \emph{beliebigen} Basisvektoren ergänzt wurde.
Es besteht also noch Spielraum für die Wahl einer \enquote{besseren} Basis, die womöglich einen Basiswechsel zu einer, im Vergleich zu einer Dreiecksmatrix, \enquote{noch diagonaleren} Matrix erlaubt.

\subsection{Blockdiagonalisierung}

Auf unserer nächsten Etappe machen wir uns auf die Suche nach einer Basis, die einen Basiswechsel zu einer \textit{Blockdiagonalmatrix} erlaubt.
Einerseits wäre das Auffinden einer solchen Basis ein Fortschritt für unsere Reise, da somit Nulleinträge auch überhalb der Diagonale entstehen würden, andererseits würde dies auch einen, wenn auch weniger großen, Rückschritt bedeuten, da wir wieder, allerdings wenige, Nicht-Nulleinträge unterhalb der Diagonale erzeugen würden.
Um das Gesamtziel unserer Reise zu erreichen, gehen wir also einen kleinen Umweg.

\subsubsection*{Invariante Untervektorräume}

An dieser Stelle rufen wir uns den Fall einer diagonalisierbaren Matrix ins Gedächtnis und denken darüber nach, welche Eigenschaft einer Basis aus Eigenvektoren eigentlich ausschlaggebend dafür ist, einen Basiswechsel zu einer Diagonalmatrix zu ermöglichen.

Sei~$\vv$ ein Eigenvektor von~$A$.
Dann wird der Untervektorraum~$U \defas \C\vv$ \emph{auf sich selbst} abgebildet, d.\,h. $A\uv \in U$ für alle $\uv \in U$.
Die Abbildung von~$U$ durch~$A$ kann also mit einem \emph{einzelnen} Basisvektor dargestellt werden.

Diese Eigenschaft einer Menge, auf sich selbst abzubilden, gestehen wir nun auch \emph{mehrdimensionalen} Untervektorräumen zu und untersuchen einen Basiswechsel mit einer ihrer Basen.

\begin{definition}[Invarianter Untervektorraum]
    Sei $A \in \C^{n \times n}$.
    Ein Untervektorraum~$U$ von~$\C^n$ heißt \textit{($A$-)invariant}, falls
    \begin{equation}
        A\uv \in U \quad \text{für alle $\uv \in U$}.
    \end{equation}
\end{definition}

\begin{remark}
    Eigenräume sind also insbesondere invariant.
\end{remark}

\begin{example}
    \label{exm:invar_block}
    Sei~$A \in \C^{3 \times 3}$ und $\vv_1,\vv_2,\vv_3 \in \C^3$ linear unabhängig, sodass
    \begin{equation}
        U \defas \linspan(\{\vv_1,\vv_2\}) \quad\text{und}\quad W \defas \linspan(\{\vv_3\})
    \end{equation}
    $A$-invariante Untervektorräume von~$\C^3$ bilden, also insbesondere $\C^3 = U \oplus W$.
    Es gilt also
    \begin{alignat}{5}
        A\vv_1 & = b_{11}\vv_1 &  & + b_{21}\vv_2                   &  &             \\
        A\vv_2 & = b_{12}\vv_1 &  & + b_{22}\vv_2 \phantom{+ \vv_3} &  &             \\
        A\vv_3 & =             &  &                                 &  & b_{33}\vv_3 \\
    \end{alignat}
    für $b_{11},b_{12},b_{21},b_{22},b_{33} \in \C$ und somit
    \begin{equation}
        A \underbrace{
            \begin{pmatrix}
                |     & |     & |     \\
                \vv_1 & \vv_2 & \vv_3 \\
                |     & |     & |
            \end{pmatrix}
        }_{\asdef V}
        =
        \begin{pmatrix}
            |     & |     & |     \\
            \vv_1 & \vv_2 & \vv_3 \\
            |     & |     & |
        \end{pmatrix}
        \underbrace{
            \begin{pmatrix}
                b_{11} & b_{12} & 0      \\
                b_{21} & b_{22} & 0      \\
                0      & 0      & b_{33}
            \end{pmatrix}
        }_{\asdef B}.
    \end{equation}
    Da~$V$ aufgrund der linearen Unabhängigkeit von $\vv_1, \vv_2, \vv_3$ invertierbar ist, gilt  $A = VBV\inv$.
    Die Matrix~$A$ ist also ähnlich zur Blockdiagonalmatrix~$B$.
\end{example}

Wir können Beispiel~\ref{exm:invar_block} problemlos verallgemeinern:
Ist $A \in \C^{n \times n}$ und
\begin{equation}
    \C^n = U_1 \oplus \cdots \oplus U_m    
\end{equation}
für $A$-invariante Untervektorräume $U_1, \ldots, U_m$, dann gibt es einen Basiswechsel zu einer Blockdiagonalmatrix
\begin{equation}
    B \defas
    \begin{pmatrix}
        \drawmatrix{B_1} &        & 0                \\
                         & \ddots &                  \\ 
        0                &        & \drawmatrix{B_m}
    \end{pmatrix}
\end{equation}
mit Blockmatrizen~$B_j \in \C^{n_j \times n_j}$ der Größe $n_j \defas \dim U_j$.
Die Umkehrung gilt dabei genauso: Gibt es einen Basiswechsel zu einer Blockdiagonalmatrix, so wird~$\C^n$ in entsprechende $A$-invariante Untervektorräume zerlegt.

\begin{remark}
    Im Fall einer diagonalisierbaren Matrix haben die Blöcke die Größe~$1$, da Eigenvektoren auf ein Vielfaches \emph{von sich selbst} abgebildet werden und die invarianten Untervektorräume demnach eindimensional sind.
\end{remark}

Die Blockstruktur wird also in der Tat mit Basen invarianter Untervektorräume erzeugt.
Um eine nicht-diagonalisierbare Matrix mittels eines Basiswechsels auf Blockdiagonalform zu bringen, müssen wir uns also auf die Suche nach einer invarianten Erweiterung von Eigenräumen begeben.

\subsubsection*{Haupträume}

Wir erinnern uns, dass bei nicht-diagonalisierbaren Matrizen~$A \in \C^{n \times n}$ zu wenig linear unabhängige Eigenvektoren zur Bildung einer Basis zur Verfügung stehen.
Demnach könnten wir versuchen, Eigenräume aufzuweichen und passend zu erweitern, sodass jedoch ihre entscheidene Eigenschaft der Invarianz erhalten bleibt.
Die Tatsache, dass
\begin{equation}
    \Eig_A(\lambda) = \nullsp(A - \lambda I) \subseteq \nullsp(A - \lambda I)^2 \subseteq \cdots \subseteq \nullsp(A - \lambda I)^n \subseteq \cdots
\end{equation}
für $\lambda \in \sigma(A)$, lässt uns Eigenräume als Nullräume von $A - \lambda I$  kanonisch erweitern, indem wir auch Potenzen von $A - \lambda I$ betrachten.

\begin{example}
    Sei
    \begin{equation}
        A \defas \begin{pmatrix}
            1 & 1 & 1  \\
            0 & 1 & 1  \\
            0 & 0 & 1
        \end{pmatrix}
    \end{equation}
    mit $\sigma(A) = \{1\}$.
    Für $\lambda \defas 1$ gilt dann
    \begin{equation}
        A - \lambda I =
        \begin{pmatrix}
            0 & 1 & 1  \\
            0 & 0 & 1  \\
            0 & 0 & 0
        \end{pmatrix}
    \end{equation}
    mit
    \begin{equation}
        \Eig_A(\lambda) = \nullsp(A - \lambda I) = \C \begin{pmatrix}1 \\ 0 \\ 0\end{pmatrix},
    \end{equation}
    sodass $\dim \Eig_A(\lambda) = 1 < 3 = n$.
    Lassen wir Potenzen bis~$n$ zu, erhalten wir
    \begin{equation}
        (A - \lambda I)^2 =
        \begin{pmatrix}
            0 & 1 & 1  \\
            0 & 0 & 1  \\
            0 & 0 & 0
        \end{pmatrix}
        \begin{pmatrix}
            0 & 1 & 1  \\
            0 & 0 & 1  \\
            0 & 0 & 0
        \end{pmatrix}
        =
        \begin{pmatrix}
            0 & 0 & 1  \\
            0 & 0 & 0  \\
            0 & 0 & 0
        \end{pmatrix}
    \end{equation}
    mit
    \begin{equation}
        \nullsp(A - \lambda I)^2 = \C \begin{pmatrix}1 \\ 0 \\ 0\end{pmatrix} + \C \begin{pmatrix}0 \\ 1 \\ 0\end{pmatrix}
    \end{equation}
    und
    \begin{equation}
        (A - \lambda I)^3 =
        \begin{pmatrix}
            0 & 1 & 1  \\
            0 & 0 & 1  \\
            0 & 0 & 0
        \end{pmatrix}
        \begin{pmatrix}
            0 & 0 & 1  \\
            0 & 0 & 0  \\
            0 & 0 & 0
        \end{pmatrix}
        =
        \begin{pmatrix}
            0 & 0 & 0  \\
            0 & 0 & 0  \\
            0 & 0 & 0
        \end{pmatrix}
    \end{equation}
    mit
    \begin{equation}
        \nullsp(A - \lambda I)^3 = \C^3.
    \end{equation}
    Wir erweitern die Nullräume von~$A - \lambda I$ hier also tatsächlich, indem wir Potenzen zulassen.
\end{example}

\begin{definition}[Hauptvektor, Hauptraum]
    Sei $A \in \C^{n \times n}$ und~$\lambda \in \sigma(A)$.
    Dann heißt die Menge
    \begin{equation}
        \Hau_A(\lambda) \defas \nullsp(A - \lambda I)^n
    \end{equation}
    \textit{Hauptraum von~$A$ zu~$\lambda$}.
    Ein Vektor $\vv \in \Hau_A(\lambda)$ mit $\vv \not= \vec{0}$ heißt \textit{Hauptvektor zu~$\lambda$ der Stufe~$k$}, $k \in \N$, falls
    \begin{equation}
        \vv \in \nullsp(A - \lambda I)^k \setminus \nullsp(A - \lambda I)^{k-1}.
    \end{equation}
\end{definition}

\begin{remark}
    \label{rem:hau}
    \begin{enumerate}[label=(\alph*)]
        \item \label{it:nullsp_pow_subset}
              Es gilt $\nullsp(A - \lambda I)^m \subseteq \nullsp(A - \lambda I)^{m+1}$ für alle $m \in \N$.
        \item $\Hau_A(\lambda)$ ist ein Untervektorraum von~$\C^n$, für den gilt
              \begin{equation}
                  1 \leq \dim\Eig_A(\lambda) \leq \dim\Hau_A(\lambda) \leq n.
              \end{equation}
        \item Eigenvektoren sind Hauptvektoren der Stufe~$1$.
    \end{enumerate}
\end{remark}

Wir müssen nun hoffen, dass Haupträume als Erweiterung von Eigenräumen genügend linear unabhängige Vektoren zur Bildung einer Basis liefern und zudem invariant sind, um einen Basiswechsel zu einer Blockdiagonalmatrix zu ermöglichen.

Wir halten zunächst fest, dass die Erweiterung des Nullraums von~$A - \lambda I$ durch Potenzbildung bei Potenz~$n$ endet.
\begin{lemma}
    \label{lem:hau_max_level}
    Sei $A \in \C^{n \times n}$ und $\lambda \in \sigma(A)$.
    Dann gilt
    \begin{equation}
        \nullsp(A - \lambda I)^m = \Hau_A(\lambda) \quad \text{für $m \geq n$}.
    \end{equation}
\end{lemma}

Außerdem können wir zeigen, dass eine Erweiterung, die einmal ins Stocken geraten ist, nicht mehr weiter fortgesetzt werden kann.
\begin{lemma}
    \label{lem:nullsp_pow_end}
    Sei $A \in \C^{n \times n}$ und $\lambda \in \sigma(A)$.
    Dann gilt: Falls
    \begin{equation}
        \nullsp(A - \lambda I)^{m+1} = \nullsp(A - \lambda I)^m \quad \text{für ein $m \in \N$},
    \end{equation}
    dann auch
    \begin{equation}
        \nullsp(A - \lambda I)^{m+k} = \nullsp(A - \lambda I)^m \quad \text{für alle $k \in \N$}.
    \end{equation}
\end{lemma}

Die folgenden zentralen strukturellen Aussagen geben nun Aufschluss darüber, dass es in der Tat die Haupträume von~$A$ sind, auf die wir unser Augenmerk zur Blockdiagonalisierung von~$A$ richten sollten.

\begin{proposition}
    \label{prop:hau_span}
    Sei $A \in \C^{n \times n}$.
    Die Hauptvektoren von~$A$ spannen~$\C^n$ auf.
\end{proposition}

\begin{proposition}
    \label{prop:hau_indep}
    Sei $A \in \C^{n \times n}$.
    Hauptvektoren zu verschiedenen Eigenwerten von~$A$ sind linear unabhängig.
\end{proposition}

\begin{theorem}
    \label{thm:hau_direct_sum}
    Sei $A \in \C^{n \times n}$.
    Dann gilt
    \begin{equation}
        \C^n = \Hau_A(\lambda_1) \oplus \cdots \oplus \Hau_A(\lambda_m),
    \end{equation}
    wobei $\lambda_1, \ldots, \lambda_m$ die verschiedenen Eigenwerte von~$A$ bezeichnen.
\end{theorem}
\begin{proof}
    Prop.~\ref{prop:hau_span} liefert
    \begin{equation}
        \C^n = \Hau_A(\lambda_1) + \cdots + \Hau_A(\lambda_m).
    \end{equation}
    Prop.~\ref{prop:hau_indep} liefert
    \begin{equation}
        \Hau_A(\lambda_1) \cap \cdots \cap \Hau_A(\lambda_m) = \{\vec{0}\}.
    \end{equation}
\end{proof}

\begin{proposition}
    \label{prop:hau_inv}
    Sei~$A \in \C^{n \times n}$.
    Die Haupträume von~$A$ sind $A$-invariant.
\end{proposition}
\begin{proof}
    Sei $\lambda \in \sigma(A)$ und $\vv \in \Hau_A(\lambda)$, d.\,h. $(A - \lambda I)^n \vv = \vec{0}$.
    Da
    \begin{equation}
        (A - \lambda I) A = A \cdot A - \lambda A = A(A - \lambda I),
    \end{equation}
    gilt auch
    \begin{equation}
        (A - \lambda I)^n A \vv = A (A - \lambda I)^n \vv = A \cdot \vec{0} = \vec{0}.
    \end{equation}
    Also ist $A\vv \in \Hau_A(\lambda)$.
\end{proof}

Da also Haupträume von~$A$ eine Basis von~$\C^n$ liefern (Satz~\ref{thm:hau_direct_sum}) und $A$-invariant sind (Prop.~\ref{prop:hau_inv}), ermöglichen deren Basen einen Basiswechsel zu einer Blockdiagonalmatrix.
Pro Eigenwert~$\lambda \in \sigma(A)$ gibt es darin einen Block der Größe $\dim \Hau_A(\lambda)$.

Wir haben das Etappenziel also erreicht.
Die Basis, die für einen Hauptraum an dieser Stelle gewählt werden kann, ist allerdings nach wie vor \emph{beliebig}, d.\,h. es gibt immer noch Spielraum für eine \enquote{bessere} Basis, die einen Basiswechsel zu einer, im Vergleich zu einer Blockdiagonalmatrix, \enquote{noch diagonaleren} Matrix erlaubt.
Das motiviert uns, unsere Reise fortzusetzen.

\subsection{Blockweise Trigonalisierung}

Die nächste Etappe soll zum Ziel haben, einen Basiswechsel zu einer \emph{blockweisen} Dreiecksmatrix zu finden.
Damit würden wir den Rückschritt, den wir in der vorherigen Etappe eingestehen mussten, wieder gutmachen, da die Nicht-Nulleinträge unterhalb der Diagonale wieder zu Nulleinträgen würden.
Dieser Umweg auf unserer Reise wäre damit abgeschlossen, sodass wir wieder auf den direkten Pfad zurückkehren können, auf den uns die erste Etappe zunächst gebracht hatte.

Da die Basen der Haupträume von~$A$ einen Basiswechsel zu einer Blockdiagonalmatrix ermöglichen, d.\,h. $A = V_H B V_H\inv$ mit
\begin{equation}
    B = \begin{pmatrix}
        \drawmatrix{B_1} &        & 0                \\
                         & \ddots &                  \\
        0                &        & \drawmatrix{B_m}
    \end{pmatrix}
\end{equation}
und einer invertierbaren Matrix $V_H \in \C^{n \times n}$, die (beliebige) Basen der Haupt\-räume enthält, können wir jeden Block als eigenständige Abbildung innerhalb eines Hauptraums betrachten und darin trigonalisieren.
Eine Trigonalisierung der Blöcke ergibt
\begin{equation}
    \begin{split}
        B & =
        \begin{pmatrix}
            \drawmatrix[width=6em]{V_1 U_1 V_1\inv} &        & 0                                       \\
                                                    & \ddots &                                         \\
            0                                       &        & \drawmatrix[width=6em]{V_m U_m V_m\inv}
        \end{pmatrix} \\
          & =
        \underbrace{
            \begin{pmatrix}
                \drawmatrix{V_1} &        & 0                \\
                                 & \ddots &                  \\
                0                &        & \drawmatrix{V_m}
            \end{pmatrix}
        }_{\asdef V_T}
        \underbrace{
            \begin{pmatrix}
                \drawmatrix{U_1} &        & 0                \\
                                 & \ddots &                  \\
                0                &        & \drawmatrix{U_m}
            \end{pmatrix}
        }_{\asdef U}
        \underbrace{
            \begin{pmatrix}
                \drawmatrix{V_1\inv} &        & 0                    \\
                                     & \ddots &                      \\
                0                    &        & \drawmatrix{V_m\inv}
            \end{pmatrix}
        }_{= V_T\inv}
    \end{split}
\end{equation}
für obere Dreiecksmatrizen $U_j \in \C^{n_j \times n_j}$ und invertierbare Matrizen $V_j \in \C^{n_j \times n_j}$.
Schließlich gilt
\begin{equation}
    \label{eq:basis_trigon_hau}
    A = V_H V_T U V_T\inv V_H\inv = V U V\inv,
\end{equation}
wobei im Produkt $V \defas V_H V_T$ die Basen der blockweisen Trigonalisierungen in den Basen der entsprechenden Haupträume ausgedrückt werden.

Die Diagonaleinträge jeder Dreiecksmatrix~$U_j$ entstammen dem Spektrum von~$B_j$ (Bem.~\ref{rem:trig_eig}), das jedoch aus nur einem einzelnen Eigenwert besteht.

\begin{proposition}
    \label{prop:hau_eig}
    Sei $A \in \C^{n \times n}$ und $\Hau_A(\lambda)$ ein Hauptraum zu~$\lambda \in \sigma(A)$.
    Eine Matrix~$B$ stelle zudem die Abbildung von $\Hau_A(\lambda)$ durch~$A$ in einer Basis von~$\Hau_A(\lambda)$ dar, d.\,h.
    \begin{equation}
        AV = VB
    \end{equation}
    für eine Matrix~$V$, die aus Basisvektoren von $\Hau_A(\lambda)$ besteht.
    Dann gilt
    \begin{equation}
        \sigma(B) = \{\lambda\}.
    \end{equation}
\end{proposition}

Die Dreiecksblöcke in~$U$ haben demnach die Form
\begin{equation}
    U_j = \begin{pmatrix}
        \lambda_j &        & *          \\
                  & \ddots &            \\
        0         &        & \lambda_j
    \end{pmatrix}.
\end{equation}
Die Häufigkeit des Auftretens von~$\lambda_j$ auf der Diagonale von~$U$ gleicht somit der Größe des Blocks~$U_j$, also der Dimension des jeweiligen Hauptraums.
Wir führen dafür einen passenden Begriff ein.

\begin{definition}[Vielfachheit]
    Sei~$A \in \C^{n \times n}$ und $\lambda \in \sigma(A)$.
    Dann heißt
    \begin{equation}
        \mult_A(\lambda) \defas \dim \Hau_A(\lambda)
    \end{equation}
    die \textit{Vielfachheit von~$\lambda$}.
\end{definition}

\begin{remark}
    \begin{enumerate}[label=(\alph*)]
        \item Die meisten Lehrbücher zu Linearer Algebra unterscheiden zwischen \textit{algebraischer} und \textit{geometrischer} Vielfachheit von Eigenwerten.
              Algebraische Vielfachheiten, d.\,h. die Vielfachheiten der Nullstellen des charakteristischen Polynoms einer Matrix, stimmen zwar, wenn man den Wert betrachtet, mit unserem Begriff der Vielfachheit überein, sollen konzeptionell hier jedoch keine Rolle spielen, da wir ohne algebraische Begriffe und Hilfsmittel auskommen möchten.
        \item Es gilt $\mult_A(\lambda) \leq n$.
        \item Mit der Zerlegung in~\eqref{eq:basis_trigon_hau} erhalten wir sofort
              \begin{equation}
                  \sum_{\lambda \in \sigma(A)} \mult_A(\lambda) = n
              \end{equation}
              und
              \begin{equation}
                  \trace(A) = \trace(V U V\inv) = \trace(V\inv V U) = \trace(U) = \sum_{\lambda \in \sigma(A)} \lambda \cdot \mult_A(\lambda).
              \end{equation}
    \end{enumerate}
\end{remark}

Damit haben wir das Etappenziel durch eine Kombination von Blockdiagonalisierung und Trigonalisierung erreicht.
Der entsprechende Basiswechsel in~\eqref{eq:basis_trigon_hau} hängt jedoch wiederum unmittelbar von den \emph{beliebigen} Basen der Haupträume für die Blockdiagonalisierung und den \emph{beliebigen} Basisergänzungen zur blockweisen Trigonalisierung ab.
Es besteht also weiterhin Spielraum, um einen Basiswechsel zu einer, im Vergleich zu einer blockweisen Dreiecksmatrix, \enquote{noch diagonaleren} Matrix zu finden, wofür wir schlussendlich \emph{konkretere} Basen der Haupträume suchen sollten.
Mit diesem durchaus ambitionierten Vorhaben begeben wir uns schließlich auf die letzte Etappe unserer Reise.

\subsection{Jordanisierung}

In dieser letzten Etappe möchten wir das Gesamtziel unserer Reise, die Jordansche Normalform einer Matrix~$A \in \C^{n \times n}$, erreichen.
Dazu stellen wir zunächst fest, dass, ausgehend von der vorherigen Etappe, die Anzahl der Nicht-Nulleinträge überhalb der Diagonale eines Dreiecksblocks ein Maß für die \enquote{Kopplung} der Basisvektoren des entsprechenden Hauptraums von~$A$ ist, da sie angibt, \emph{wie viele} Basisvektoren zur Darstellung der Abbildung eines bestimmten Basisvektors durch~$A$ benötigt werden.
Wir können also versuchen, die Basisvektoren der Haupträume \emph{geschickt} aufeinander aufbauen zu lassen, sodass deren Kopplung möglichst lose wird.
Dadurch würde die Anzahl der Nicht-Nulleinträge überhalb der Diagonale der Dreiecksblöcke deutlich verringert.

Da Haupträume von~$A$ invariant sind, ist die Behandlung eines jeden Blocks bzw. Eigenwerts gleich.
Es reicht also aus, wenn wir uns im Folgenden auf einen einzelnen Block bzw. Eigenwert konzentrieren.

Sei also~$\lambda \in \sigma(A)$.
Zunächst definieren wir
\begin{equation}
    L \defas \min \setbr{k \in \N}{\nullsp(A - \lambda I)^{k+1} = \nullsp(A - \lambda I)^k} \leq n
\end{equation}
als die höchstmögliche Stufe von Hauptvektoren zu~$\lambda$.
Dann gelten die Teilmengenrelationen
\begin{equation}
    \label{eq:cascade_eig_hau}
    \Eig_A(\lambda) = \nullsp(A - \lambda I) \subset \nullsp(A - \lambda I)^2 \subset \cdots \subset \nullsp(A - \lambda I)^L = \Hau_A(\lambda).
\end{equation}
Beginnend mit einem Hauptvektor~$\vv_L$ der Stufe~$L$ erzeugen wir nun eine Folge von, wie wir sehen werden, linear unabhängigen Hauptvektoren~$\vv_k$ abfallender Stufen~$k \leq L$, die die gesamte Kaskade an Nullräumen in~\eqref{eq:cascade_eig_hau} ausschöpft, indem wir in jedem Schritt mit~$A - \lambda I$ multiplizieren:
\begin{gather}
    \vv_L \in \nullsp(A - \lambda I)^L \setminus \nullsp(A - \lambda I)^{L-1} \\
    (A - \lambda I)^L \vv_L = \vec{0}
\end{gather}
\vspace{-2\baselineskip}
\begin{center}
    $\big\downarrow$
\end{center}
\vspace{-\baselineskip}
\begin{gather}
    (A - \lambda I)^{L-1} (A - \lambda I) \vv_L = \vec{0} \\
    \vv_{L-1} \defas (A - \lambda I) \vv_L \in \nullsp(A - \lambda I)^{L-1} \setminus \nullsp(A - \lambda I)^{L-2} \\
    (A - \lambda I)^{L-1} \vv_{L-1} = \vec{0}
\end{gather}
\vspace{-2\baselineskip}
\begin{center}
    $\big\downarrow$
\end{center}
\vspace{-\baselineskip}
\begin{gather}
    (A - \lambda I)^{L-2} (A - \lambda I) \vv_{L-1} = \vec{0} \\
    \vv_{L-2} \defas (A - \lambda I) \vv_{L-1} \in \nullsp(A - \lambda I)^{L-2} \setminus \nullsp(A - \lambda I)^{L-3} \\
    (A - \lambda I)^{L-2} \vv_{L-2} = \vec{0}
\end{gather}
\vspace{-2\baselineskip}
\begin{center}
    $\big\downarrow$ \\
    $\vdots$ \\
    $\big\downarrow$
\end{center}
\vspace{-\baselineskip}
\begin{gather}
    (A - \lambda I)^{L-(L-1)} (A - \lambda I) \vv_{L-(L-2)} = \vec{0} \\
    \vv_1 \defas (A - \lambda I) \vv_2 \in \nullsp(A - \lambda I) \setminus \{\vec{0}\} \\
    (A - \lambda I) \vv_1 = \vec{0}
\end{gather}
Tatsächlich ist $\vv_k \not\in \nullsp(A - \lambda I)^{k-1}$, da $\vv_k = (A - \lambda I)^{L-k} \vv_L$ und
\begin{equation}
    (A - \lambda I)^{k-1} \vv_k = (A - \lambda I)^{k-1} (A - \lambda I)^{L-k} \vv_L = (A - \lambda I)^{L-1} \vv_L \not= \vec{0}.
\end{equation}
Per Konstruktion gilt zudem
\begin{alignat}{3}
                          &  & \vv_{k-1} & = (A - \lambda I) \vv_k       \\
    \label{eq:jordan_map}
    \Leftrightarrow \quad &  & A\vv_k    & = \vv_{k-1} + \lambda \vv_k.
\end{alignat}
Wir beobachten, dass \eqref{eq:jordan_map} der Eigenwertgleichung in~\eqref{eq:eig} stark ähnelt.
In Matrix-Schreibweise ergibt sich somit insgesamt
\begin{equation}
    \label{eq:jord_chain_block}
    A \begin{pmatrix}
        |     &        & |     \\
        \vv_1 & \cdots & \vv_L \\
        |     &        & |
    \end{pmatrix}
    =
    \begin{pmatrix}
        |     &        & |     \\
        \vv_1 & \cdots & \vv_L \\
        |     &        & |
    \end{pmatrix}
    \begin{pmatrix}
        \lambda & 1       &        & 0       \\
                & \lambda & \ddots &         \\
                &         & \ddots & 1       \\
        0       &         &        & \lambda \\
    \end{pmatrix}.
\end{equation}
Eine solche Folge $\vv_\ell, \ldots, \vv_1$ von Hauptvektoren verdient einen naheliegenden Begriff.

\begin{definition}[Jordan-Kette]
    Sei $A \in \C^{n \times n}$ und $\lambda \in \sigma(A)$.
    Eine Folge von Hauptvektoren $\vv_\ell, \ldots, \vv_1$ zu~$\lambda$ abfallender Stufe mit $\ell \leq L$, sodass
    \begin{equation}
        \vv_{k-1} = (A - \lambda I) \vv_k \quad \text{für $k = 2, \ldots, \ell$}
    \end{equation}
    heißt \textit{Jordan-Kette zu~$\lambda$}.
\end{definition}

\begin{remark}
    Die Kopplung von Hauptvektoren in einer Jordan-Kette ist demnach denkbar lose.
    In der Tat, die Abbildung jedes Kettenglieds~$\vv_k$ durch~$A$ lässt sich mit lediglich höchstens zwei Kettengliedern, nämlich~$\vv_{k-1}$ und~$\vv_k$, darstellen.
\end{remark}

\begin{example}
    Sei
    \begin{equation}
        A \defas \begin{pmatrix}
            2  & 1 & 1  \\
            -4 & 5 & 4  \\
            1  & 0 & 2
        \end{pmatrix}
    \end{equation}
    mit $\sigma(A) = \{3\}$.
    Für $\lambda \defas 3$ gilt
    \begin{equation}
        A - \lambda I =
        \begin{pmatrix}
            -1 & 1 & 1   \\
            -4 & 2 & 4   \\
            1  & 0 & -1
        \end{pmatrix},
        \quad
        (A - \lambda I)^2 =
        \begin{pmatrix}
            -2 & 1 & 2  \\
            0  & 0 & 0  \\
            -2 & 1 & 2
        \end{pmatrix},
        \quad
        (A - \lambda I)^3 =
        \begin{pmatrix}
            0 & 0 & 0  \\
            0 & 0 & 0  \\
            0 & 0 & 0
        \end{pmatrix}
    \end{equation}
    und
    \begin{equation}
        \nullsp(A - \lambda I) = \C \begin{pmatrix}1 \\ 0 \\ 1\end{pmatrix},\quad
        \nullsp(A - \lambda I)^2 = \C \begin{pmatrix}1 \\ 0 \\ 1\end{pmatrix} + \C \begin{pmatrix}1 \\ 2 \\ 0\end{pmatrix}
        ,\quad
        \nullsp(A - \lambda I)^3 = \C^3.
    \end{equation}
    
    Wir beginnen mit einem Hauptvektor der höchsten Stufe $L=3$
    \begin{gather}
        \vv_3 \defas \begin{pmatrix}1 \\ 0 \\ 0\end{pmatrix} \in \nullsp(A - \lambda I)^3 \setminus \nullsp(A - \lambda I)^2
    \end{gather}
    und erzeugen die zugehörige Jordan-Kette
    \begin{align}
        \vv_2 & = (A - \lambda I) \vv_3 = \begin{pmatrix}-1 \\ -4 \\ 1\end{pmatrix},  \\
        \vv_1 & = (A - \lambda I) \vv_2 = \begin{pmatrix}-2 \\ 0 \\ -2\end{pmatrix}.
    \end{align}
    
    Damit gilt
    \begin{equation}
        A \begin{pmatrix}
            -2 & -1 & 1  \\
            0  & -4 & 0  \\
            -2 & 1  & 0
        \end{pmatrix}
        =
        \begin{pmatrix}
            -2 & -1 & 1  \\
            0  & -4 & 0  \\
            -2 & 1  & 0
        \end{pmatrix}
        \begin{pmatrix}
            3 & 1 & 0  \\
            0 & 3 & 1  \\
            0 & 0 & 3
        \end{pmatrix}.
    \end{equation}
\end{example}

Eine einzelne Jordan-Kette ist linear unabhängig per Konstruktion, da deren Glieder aus unterschiedlichen Bereichen des entsprechenden Hauptraums stammen.

\begin{proposition}
    \label{prop:jord_chain_indep}
    Sei~$A \in \C^{n \times n}$ und~$\lambda \in \sigma(A)$.
    Eine Jordan-Kette zu~$\lambda$ ist linear unabhängig.
\end{proposition}

In~\eqref{eq:jord_chain_block} können wir zudem beobachten, dass eine einzelne Jordan-Kette, die noch keine Basis des entsprechenden Hauptraums bildet, bereits zu einem Block führt.

\begin{proposition}
    \label{prop:jord_chain_inv}
    Sei $A \in \C^{n \times n}$ und $\lambda \in \sigma(A)$.
    Eine Jordan-Kette zu~$\lambda$ spannt einen $A$-invarianten Untervektorraum auf.
\end{proposition}

Neben der ersten Jordan-Kette der Länge~$L$, die bei einem Hauptvektor der höchsten Stufe~$L$ beginnt, kann es noch weitere, wie wir sehen werden, linear unabhängige Jordan-Ketten zu~$\lambda$ geben, welche jedoch im allgemeinen bei Hauptvektoren einer Stufe~$\ell \leq L$ beginnen und somit Länge~$\ell$ haben.
Über die Längen dieser weiteren Jordan-Ketten zu~$\lambda$ ist keine allgemeine Aussage möglich; sie hängt von der Anzahl linear unabhängiger Hauptvektoren der einzelnen Stufen, und damit von~$A$, ab.

Ausgehend von bereits erzeugten Jordan-Ketten wird eine weitere Jordan-Kette zu~$\lambda$ erzeugt, indem man einen Hauptvektor zu~$\lambda$ höchstmöglicher Stufe findet, der allerdings von den Hauptvektoren zu~$\lambda$ gleicher Stufe, die bereits in einer Jordan-Kette verwendet wurden, linear unabhängig ist.
Bei diesem Hauptvektor beginnt dann die weitere Jordan-Kette zu~$\lambda$ und wird, wie gehabt, durch schrittweise Multiplikation mit $A - \lambda I$ bis zu einem Eigenvektor fortgesetzt.

Wurden alle Jordan-Ketten zu~$\lambda$ erzeugt, steht uns an dieser Stelle eine Anzahl von $\dim \Hau_A(\lambda)$ Hauptvektoren als vielversprechende Kandidaten zur Bildung der gesuchten lose gekoppelten Basis von~$\Hau_A(\lambda)$ zur Verfügung; lediglich deren lineare Unabhängigkeit steht noch im Raum.

Wir haben bereits gesehen, dass eine einzelne Jordan-Kette zu~$\lambda$ linear unabhängig ist (Prop.~\ref{prop:jord_chain_indep}).
Wir können diese Aussage jedoch, ohne größere Hürden, auf \emph{disjunkte} Jordan-Ketten zu~$\lambda$, die bei linear unabhängigen Hauptvektoren zu~$\lambda$ beginnen, erweitern.
Die Konkretisierung mittels des Beiworts \enquote{disjunkt} ist notwendig, um auszuschließen, dass Jordan-Ketten zusammen mit Verkürzungen, d.\,h. abgeschnittenen Teilen, von ihnen betrachtet werden.

\begin{proposition}
    Sei $A \in \C^{n \times n}$ und $\lambda \in \sigma(A)$.
    Disjunkte Jordan-Ketten zu~$\lambda$, die bei linear unabhängigen Hauptvektoren beginnen, sind linear unabhängig.
\end{proposition}
\begin{proof}
    Analog zum Beweis von Prop.~\ref{prop:jord_chain_indep}.
\end{proof}

\begin{remark}
    \begin{enumerate}[label=(\alph*)]
        \item Die Anzahl linear unabhängiger Hauptvektoren zu~$\lambda$ der Stufe~$\ell \leq L$ ist
              \begin{equation}
                  \dim \nullsp(A - \lambda I)^\ell - \dim \nullsp(A - \lambda I)^{\ell-1}.
              \end{equation}
        \item Die Anzahl linear unabhängiger Jordan-Ketten zu~$\lambda$ ist $\dim \Eig_A(\lambda)$.
        \item Die Summe der Längen linear unabhängiger Jordan-Ketten zu~$\lambda$ ist $\dim \Hau_A(\lambda)$.
    \end{enumerate}
\end{remark}

Jordan-Ketten zu~$\lambda$ bilden also tatsächlich die gesuchte lose gekoppelte Basis des Hauptraums zu~$\lambda$.
Enthält die Matrix~$V_\lambda$ eine solche Basis aus Jordan-Ketten zu~$\lambda$, dann gilt
\begin{equation}
    A V_\lambda = V_\lambda J_\lambda    
\end{equation}
für die Blockmatrix
\begin{equation}
    J_\lambda
    =
    \begin{pmatrix}
        \drawmatrix{J_{\lambda,1}} &        & 0                                   \\
                                   & \ddots &                                     \\
        0                          &        & \drawmatrix{J_{\lambda, d_\lambda}}
    \end{pmatrix}
\end{equation}
mit $d_\lambda \defas \dim \Eig_A(\lambda)$, wobei jeder Jordan-Kette zu~$\lambda$ ein \textit{Jordan-Block zu~$\lambda$}
\begin{equation}
    J_{\lambda,*} \defas
    \begin{pmatrix}
        \lambda & 1       &        & 0        \\
                & \lambda & \ddots &          \\
                &         & \ddots & 1        \\
        0       &         &        & \lambda
    \end{pmatrix}
\end{equation}
der Größe ihrer Länge entspricht.

\begin{example}
    Sei~$A$ eine Matrix mit einem einzelnen Eigenwert~$\lambda \in \C$ und folgenden Jordan-Ketten zu~$\lambda$:
    \begin{equation*}
        \renewcommand{\arraystretch}{1.5}
        \begin{array}{c|ccc}
                           & \text{Jordan-Kette~1} & \text{Jordan-Kette~2} & \text{Jordan-Kette~3} \\ \hline
            \text{Stufe 5} & \vv_5^{(1)}           &                       &                       \\
                           & \downarrow            &                       &                       \\
            \text{Stufe 4} & \vv_4^{(1)}           &                       &                       \\
                           & \downarrow            &                       &                       \\
            \text{Stufe 3} & \vv_3^{(1)}           & \vv_3^{(2)}           &                       \\
                           & \downarrow            & \downarrow            &                       \\
            \text{Stufe 2} & \vv_2^{(1)}           & \vv_2^{(2)}           & \vv_2^{(3)}           \\
                           & \downarrow            & \downarrow            & \downarrow            \\
            \text{Stufe 1} & \vv_1^{(1)}           & \vv_1^{(2)}           & \vv_1^{(3)}
        \end{array}
    \end{equation*}
    Dann gilt
    \begin{equation*}
        J_\lambda = \left(
        \begin{array}{cccccccccc}
            \cline{1-5}
            \multicolumn{1}{|c}{\lambda} & 1       &         &         & \multicolumn{1}{c|}{0}        &                              &         &                              &                              &                              \\
            \multicolumn{1}{|c}{}        & \lambda & 1       &         & \multicolumn{1}{c|}{}         &                              &         &                              & 0                            &                              \\
            \multicolumn{1}{|c}{}        &         & \lambda & 1       & \multicolumn{1}{c|}{}         &                              &         &                              &                                                             \\
            \multicolumn{1}{|c}{}        &         &         & \lambda & \multicolumn{1}{c|}{1}        &                              &         &                              &                                                             \\
            \multicolumn{1}{|c}{0}       &         &         &         & \multicolumn{1}{c|}{\lambda } &                              &         &                              &                                                             \\ \cline{1-8}
                                         &         &         &         &                               & \multicolumn{1}{|c}{\lambda} & 1       & \multicolumn{1}{c|}{0}       &                              &                              \\
                                         &         &         &         &                               & \multicolumn{1}{|c}{}        & \lambda & \multicolumn{1}{c|}{1}       &                              &                              \\
                                         &         &         &         &                               & \multicolumn{1}{|c}{0}       &         & \multicolumn{1}{c|}{\lambda} &                              &                              \\ \cline{6-10}
                                         & 0       &         &         &                               &                              &         &                              & \multicolumn{1}{|c}{\lambda} & \multicolumn{1}{c|}{1}       \\
                                         &         &         &         &                               &                              &         &                              & \multicolumn{1}{|c}{0}       & \multicolumn{1}{c|}{\lambda} \\ \cline{9-10}
        \end{array}
        \right).
    \end{equation*}
\end{example}

Wir betrachten nun die Gesamtsituation mit allen Eigenwerten von~$A$.
Die Blockmatrizen~$J_{\lambda_j}$ der verschiedenen Eigenwerte $\lambda_j$ von~$A$ setzen sich, aufgrund der Invarianz der Haupträume von~$A$, \emph{blockweise} zusammen zu einer \textit{Jordan-Matrix}
\begin{equation}
    J \defas \begin{pmatrix}
        \drawmatrix{J_{\lambda_1}} &        & 0                          \\
                                   & \ddots &                            \\
        0                          &        & \drawmatrix{J_{\lambda_m}}
    \end{pmatrix},
\end{equation}
welche als \textit{Jordansche Normalform von~$A$} bezeichnet wird.
Die Jordansche Normalform von~$A$ hat also die Eigenwerte von~$A$ gemäß ihrer Vielfachheiten auf der Diagonale und durch Nullen getrennte Ketten von Einsen verschiedener Längen auf der oberen Nebendiagonale.

Jede Matrix~$A \in \C^{n \times n}$ ist somit ähnlich zu einer Jordan-Matrix~$J \in \C^{n \times n}$.
Die entsprechende Zerlegung
\begin{equation}
    A = V J V\inv,
\end{equation}
wobei die invertierbare (jedoch nicht eindeutige) Matrix~$V \in \C^{n \times n}$ passende Jordan-Ketten als Basen der Haupträume von~$A$ enthält, heißt \textit{Jordan-Zerlegung}.

Hiermit sind wir am Ende unserer Reise angekommen.
Es sind also die Jordan-Ketten zu Eigenwerten von~$A$, die einen Basiswechsel auf eine \enquote{möglichst diagonale} Matrix ermöglichen und uns somit Einsicht in die Geometrie von~$A$ gewähren.
Ein Versuch, die Reise in der Hoffnung auf einen \enquote{noch besseren} Basiswechsel zu einer, im Vergleich zu einer Jordan-Matrix, \enquote{noch diagonaleren} Matrix fortzusetzen, ist vergebens, denn die Jordansche Normalform ist in der Tat die \enquote{diagonalste} aller möglichen.

\subsection*{Reisenachbereitung}

In Tab.~\ref{tab:analogy_termi} möchte der Autor im Nachgang unserer Reise noch auf Analogien von Begriffen, die während der Reise eingeführt wurden, hinweisen und sie so als Reiseerfahrung für die Gereisten festhalten.

\begin{table}
    \centering
    \begin{tabular}{ccccc}
        Trigonalisierung & --- & Schur-Zerlegung  & --- & Schursche Normalform   \\
        $\updownarrow$   &     & $\updownarrow$   &     & $\updownarrow$         \\
        Jordanisierung   & --- & Jordan-Zerlegung & --- & Jordansche Normalform
    \end{tabular}
    \caption{Analogien von Begriffen als Reiseerfahrung.}
    \label{tab:analogy_termi}
\end{table}

\section{Abschluss}

Abschließend lassen wir die Erlebnisse unserer Reise noch einmal Revue passieren.
Nachdem wir anfangs die Existenz von nicht-diagonalisierbaren Matrizen festgestellt hatten, haben wir, in der Hoffnung auf einen Basiswechsel, der uns die Geometrie solcher Matrizen dennoch bestmöglich verstehen lässt, den Aufbruch zur Jordanschen Normalform gewagt.
Die erste Etappe lieferte uns bereits eine Trigonalisierung, die durch die Existenz eines Eigenwerts und die Eigenschaft eines zugehörigen Eigenvektors, auf ein Vielfaches von sich selbst abzubilden, immerhin schrittweise Nullen unterhalb der Diagonale erzeugt hat.
Der Begriff der Invarianz und seine geometrische Bedeutung für eine Blockdiagonalisierung hat uns in der zweiten Etappe Haupträume als kanonische invariante Erweiterung von Eigenräumen erleben und mit einem Beispiel einprägen lassen.
Nach einem kurzen Umweg, den wir dafür gehen mussten, wurden wir in der dritten Etappe durch eine blockweise Trigonalisierung mit den Erfahrungen der ersten beiden Etappen zurück auf den richtigen Pfad geführt, sodass wir den Begriff der Vielfachheit und seinen Zusammenhang mit der Spur einer Matrix auf natürliche Weise einführen konnten.
Auf der vierten und letzten Etappe konnten wir schließlich konkretere \enquote{lose gekoppelte} Basen von Haupträumen in Form von Jordan-Ketten als Folge von Hauptvektoren abfallender Stufe ausfindig machen, die uns letztendlich an das Ziel unserer Reise brachten und einen krönenden Abschluss derselbigen darstellten.

\begin{appendices}

    \section{Beweise}
    \label{sec:app_proofs}
    
    Die Beweise von
    \begin{center}    
        \begin{itemize*}
            \item Satz~\ref{thm:ev_exist}
        \end{itemize*}
        \quad
        \begin{itemize*}
            \item Lemma~\ref{lem:hau_max_level}
        \end{itemize*}
        \quad
        \begin{itemize*}
            \item Prop.~\ref{prop:hau_span}
        \end{itemize*}
        \quad
        \begin{itemize*}
            \item Prop.~\ref{prop:hau_indep}
        \end{itemize*}
        \quad
        \begin{itemize*}
            \item Prop.~\ref{prop:hau_eig}
        \end{itemize*}
    \end{center}
    wurden Texten von Sheldon Axler~\cite{axler1995down,axler2024linear} entnommen und auf den mathematisch angewandteren Kontext dieses Artikels angepasst.
    Insbesondere für den Beweis von Prop.~\ref{prop:hau_span} musste etwas mehr Hand angelegt werden.
    
    \subsection{Satz~\ref{thm:ev_exist}}
    \label{ssec:app_proof_thm_ev_exist}
    
    \begin{proof}
        Sei~$\vv \in \C^n$ mit $\vv \not= \vec{0}$.
        Die Vektoren $\vv, A\vv, A^2\vv, \ldots, A^n\vv$ müssen linear abhängig sein, da~$n+1$ Vektoren in einem $n$-dimensionalen Vektorraum nicht linear unabhängig sein können.
        Es gibt also eine Linearkombination
        \begin{equation}
            \alpha_0\vv + \alpha_1 A\vv + \cdots + \alpha_n A^n\vv = \vec{0},
        \end{equation}
        in der nicht alle~Koeffizienten~$\alpha_j \in \C$ gleich Null sind.
        Insbesondere ist im Fall $\alpha_0 \not= 0$ noch mindestens ein weiterer Koeffizient ungleich Null, da $\vv \not= \vec{0}$.
        Wir definieren das (nicht-konstante) komplexe Polynom
        \begin{equation}
            P(z) \defas \alpha_0 + \alpha_1 z + \cdots + \alpha_n z^n,
        \end{equation}
        sodass $P(A)\vv = \vec{0}$.
        
        Nach dem Fundamentalsatz der Algebra zerfällt~$P$ vollständig in Linearfaktoren, d.\,h.
        \begin{equation}
            P(z) = c (z - \lambda_1) \cdots (z - \lambda_d)
        \end{equation}
        für $d \defas \deg P \geq 1$, $c \defas \alpha_d \not= 0$ und $\lambda_1, \ldots, \lambda_d \in \C$.
        Es gilt also
        \begin{equation}
            c (A - \lambda_1 I) \cdots (A - \lambda_d I) \vv = P(A)\vv = \vec{0}.
        \end{equation}
        Da $c \not= 0$ und $\vv \not= \vec{0}$, muss es somit ein~$\wv \not= \vec{0}$ und $\ell \in \{1, \ldots, d\}$ geben, sodass
        \begin{equation}
            (A - \lambda_\ell I)\wv = \vec{0}.
        \end{equation}
        Der Vektor~$\wv$ ist dann Eigenvektor zum Eigenwert~$\lambda_\ell$.
    \end{proof}
    
    \subsection{Lemma~\ref{lem:hau_max_level}}
    
    \begin{proof}
        Sei $\vv \in \Hau_A(\lambda)$, d.\,h. $(A - \lambda I)^n \vv = \vec{0}$.
        Da
        \begin{equation}
            (A - \lambda I)^m \vv = (A - \lambda I)^{m-n} (A - \lambda I)^n \vv = (A - \lambda I)^{m-n} \cdot \vec{0} = \vec{0},
        \end{equation}
        ist $\vv \in \nullsp(A - \lambda I)^m$.
        Also ist $\Hau_A(\lambda) \subseteq \nullsp(A - \lambda I)^m$.
        
        Sei nun $\vv \in \nullsp(A - \lambda I)^m$ mit $\vv \not= \vec{0}$ und $k \leq m$ die Stufe von~$\vv$, d.\,h. die kleinste natürliche Zahl, sodass $(A - \lambda I)^k \vv = \vec{0}$.
        Wir zeigen, dass $k \leq n$, indem wir die lineare Unabhängigkeit von
        \begin{equation}
            \label{eq:lem_max_level_pow}
            \vv, (A - \lambda I) \vv, \ldots, (A - \lambda I)^{k-1} \vv
        \end{equation}
        nachweisen.
        Seien also $\alpha_0, \alpha_1, \ldots, \alpha_{k-1} \in \C$, sodass
        \begin{equation}
            \label{eq:lem_max_level_pow_indep}
            \alpha_0 \vv + \alpha_1 (A - \lambda I) \vv + \cdots + \alpha_{k-1} (A - \lambda I)^{k-1} \vv = \vec{0}.
        \end{equation}
        Eine Multiplikation dieser Gleichung mit $(A - \lambda I)^{k-1}$ ergibt
        \begin{alignat}{3}
                                  &  & \alpha_0 (A - \lambda I)^{k-1} \vv + \alpha_1 (A - \lambda I)^k \vv + \cdots + \alpha_{k-1} (A - \lambda I)^{2(k-1)} \vv & = \vec{0}   \\
            \Leftrightarrow \quad &  & \alpha_0 (A - \lambda I)^{k-1} \vv                                                                                       & = \vec{0}.
        \end{alignat}
        Da $(A - \lambda I)^{k-1} \vv \not= \vec{0}$, folgt $\alpha_0 = 0$.
        
        Analog können wir durch weitere Multiplikationen von~\eqref{eq:lem_max_level_pow_indep} mit
        \begin{equation}
            (A - \lambda I)^{k-2}, \ldots, (A - \lambda I), I
        \end{equation}
        zeigen, dass auch $\alpha_1 = \cdots = \alpha_{k-1} = 0$.
        Also haben wir in~\eqref{eq:lem_max_level_pow} $k$~linear unabhängige Vektoren in einem $n$-dimensionalen Vektorraum, woraus $k \leq n$ folgt.
        Somit ist $\vv \in \nullsp(A - \lambda I)^n = \Hau_A(\lambda)$ und $\nullsp(A - \lambda I)^m \subseteq \Hau_A(\lambda)$.
    \end{proof}
    
    \subsection{Lemma~\ref{lem:nullsp_pow_end}}
    
    \begin{proof}
        Da die Aussage für~$k=1$ trivial ist, nehmen wir~$k \geq 2$ an.
        
        Aus Bem.~\ref{rem:hau}\ref{it:nullsp_pow_subset} folgt bereits $\nullsp(A - \lambda I)^m \subseteq \nullsp(A - \lambda I)^{m+k}$.
        
        Sei also $\vv \in \nullsp(A - \lambda I)^{m+k}$, d.\,h. $(A - \lambda I)^{m+k} \vv = \vec{0}$.
        Falls nun $\nullsp(A - \lambda I)^{m+1} = \nullsp(A - \lambda I)^m$, dann gilt
        \begin{equation}
            \begin{split}
                \vec{0} & = (A - \lambda I)^{m+k} \vv                       \\
                        & = (A - \lambda I)^{m+1} (A - \lambda I)^{k-1} \vv \\
                        & = (A - \lambda I)^m (A - \lambda I)^{k-1} \vv     \\
                        & = (A - \lambda I)^{m+k-1} \vv                     \\
                        & = \cdots                                          \\
                        & = (A - \lambda I)^{m} \vv.
            \end{split}
        \end{equation}
        Somit ist $\vv \in \nullsp(A - \lambda I)^{m}$ und $\nullsp(A - \lambda I)^{m+k} \subseteq \nullsp(A - \lambda I)^m$.
    \end{proof}
    
    \subsection{Proposition~\ref{prop:hau_span}}
    
    \begin{proof}
        Sei $\lambda \in \sigma(A)$ (Satz~\ref{thm:ev_exist}).
        
        Wir zeigen zunächst, dass
        \begin{equation}
            \label{eq:Cn_hau_colsp}
            \C^n = \nullsp(A - \lambda I)^n \oplus \colsp(A - \lambda I)^n.
        \end{equation}
        Denn dann müssen wir nur noch zeigen, dass Hauptvektoren von~$A$ auch $\colsp(A - \lambda I)^n$ aufspannen.
        Sei also $\vv \in \nullsp(A - \lambda I)^n \cap \colsp(A - \lambda I)^n$.
        Da $\vv \in \colsp(A - \lambda I)^n$, gibt es ein $\wv \in \C^n$, sodass $(A - \lambda I)^n \wv = \vv$.
        Mit Lemma~\ref{lem:hau_max_level} gilt dann
        \begin{alignat}{3}
                                  &  & (A - \lambda I)^n \vv    & = \vec{0}   \\
            \Leftrightarrow \quad &  & (A - \lambda I)^{2n} \wv & = \vec{0}   \\
            \Leftrightarrow \quad &  & (A - \lambda I)^{n} \wv  & = \vec{0}   \\
            \Leftrightarrow \quad &  & \vv                      & = \vec{0},
        \end{alignat}
        also $\nullsp(A - \lambda I)^n \cap \colsp(A - \lambda I)^n = \{\vec{0}\}$.
        Der Rangsatz liefert schließlich~\eqref{eq:Cn_hau_colsp}.
        
        Wir zeigen nun, dass Hauptvektoren von~$A$ auch $U \defas \colsp(A - \lambda I)^n$ aufspannen.
        Mit dem Rangsatz gilt zunächst, dass $k \defas \dim U < n$, da~$\lambda \in \sigma(A)$ und somit $\dim \nullsp(A - \lambda I)^n \geq 1$.
        Der Untervektorraum~$U$ ist zudem $A$-invariant.
        In der Tat, für alle $\uv \in U$ gibt es ein $\wv \in \C^n$, sodass $(A - \lambda I)^n \wv = \uv$ und es gilt
        \begin{equation}
            A\uv = A (A - \lambda I)^n \wv = (A - \lambda I)^n A\wv \in U.
        \end{equation}
        Sei nun $B \in \C^{n \times k}$, bestehend aus Basisvektoren von~$U$.
        Dann gibt es aufgrund der Invarianz von~$U$ ein $A' \in \C^{k \times k}$, sodass
        \begin{equation}
            AB = BA'.
        \end{equation}
        Wenn die zu beweisende Aussage bereits für $k < n$ stimmte, dann stimmt sie auch für~$n$, da~$U$ dann durch Hauptvektoren von~$A$ aufgespannt würde.
        In der Tat, seien $\wv_1, \ldots, \wv_m \in \C^k$, $m \in \N$, Hauptvektoren von~$A'$ zu $\lambda'_1, \ldots, \lambda'_m \in \sigma(A')$, sodass $\linspan(\{\wv_1, \ldots, \wv_m\}) = \C^k$.
        Dann sind
        \begin{equation}
            \uv_j \defas B\wv_j \in U, \quad j = 1, \ldots, m,
        \end{equation}
        Hauptvektoren von~$A$ mit $\linspan(\{\uv_1, \ldots, \uv_m\}) = U$.
        Denn einerseits gilt
        \begin{equation}
            (A - \lambda' I_n) B = AB - \lambda' B = BA' - \lambda' B = B (A' - \lambda' I_k) \quad \text{für $\lambda' \in \C$},
        \end{equation}
        womit
        \begin{equation}
            (A - \lambda_j' I)^n \uv_j = (A - \lambda_j' I)^n B\wv_j = B (A' - \lambda_j' I)^n \wv_j = \vec{0}, \quad \text{$j = 1, \ldots, m$},
        \end{equation}
        und andererseits gibt es für alle~$\uv \in U$ ein
        \begin{equation}
            \wv = \alpha_1 \wv_1 + \cdots + \alpha_m \wv_m \quad \text{mit $\alpha_1, \ldots, \alpha_m \in \C$},
        \end{equation}
        sodass
        \begin{equation}
            \uv = B\wv = \alpha_1 B\wv_1 + \cdots + \alpha_m B\wv_m = \alpha_1 \uv_1 + \cdots + \alpha_m \uv_m.
        \end{equation}
        
        Wir können die zu beweisende Aussage also per Induktion über~$n$ zeigen.
        Der Induktionsanfang $n=1$ ist dabei trivial.
    \end{proof}
    
    \subsection{Proposition~\ref{prop:hau_indep}}
    
    \begin{proof}
        Seien $\lambda_1, \ldots, \lambda_m$ die verschiedenen Eigenwerte von~$A$ und $\vv_1 \in \Hau_A(\lambda_1), \ldots, \vv_m \in \Hau_A(\lambda_m)$ zugehörige Hauptvektoren.
        Es gelte
        \begin{equation}
            \label{eq:prop_hau_indep_pow_indep}
            \alpha_1\vv_1 + \cdots + \alpha_m\vv_m = \vec{0} \quad \text{für $\alpha_1, \ldots, \alpha_m \in \C$}.
        \end{equation}
        Wir müssen zeigen, dass~$\alpha_1 = \cdots = \alpha_m = 0$.
        
        Sei~$k_1 \in \N$ die Stufe von~$\vv_1$, d.\,h. die kleinste natürliche Zahl, sodass $(A - \lambda_1 I)^{k_1} \vv_1 = \vec{0}$.
        Dann ergibt eine Multiplikation von~\eqref{eq:prop_hau_indep_pow_indep} mit $(A - \lambda_1 I)^{k_1-1} (A - \lambda_2 I)^n \cdots (A - \lambda_m I)^n$, dass
        \begin{equation}
            \label{eq:hau_indep_mult_powers}
            \alpha_1 (A - \lambda_1 I)^{k_1-1} (A - \lambda_2 I)^n \cdots (A - \lambda_m I)^n \vv_1 = \vec{0}.
        \end{equation}
        Wenn wir jeden Term in $(A - \lambda_2 I)^n \cdots (A - \lambda_m I)^n$ mittels des Binomischen Lehrsatzes zu
        \begin{equation}
            \begin{split}
                (A - \lambda_j I)^n & = ((A - \lambda_1 I) + (\lambda_1 - \lambda_j) I)^n                                                 \\
                                    & = \sum_{k=0}^n \begin{pmatrix}n \\ k\end{pmatrix} (A - \lambda_1 I)^k (\lambda_1 - \lambda_j)^{n-k}
            \end{split}
        \end{equation}
        umformen, erhalten wir in~\eqref{eq:hau_indep_mult_powers} eine Summe aus Termen.
        Bis auf den Term
        \begin{equation}
            (\lambda_1 - \lambda_2)^n \cdots (\lambda_1 - \lambda_m)^n
        \end{equation}
        enthält jeder dieser Terme eine Potenz von $A - \lambda_1 I$, die kombiniert mit $(A - \lambda_1 I)^{k_1-1} \vv_1$ den Nullvektor ergibt.
        
        Die Gleichheit in~\eqref{eq:hau_indep_mult_powers} wird also zu
        \begin{equation}
            \alpha_1 (\lambda_1 - \lambda_2)^n \cdots (\lambda_1 - \lambda_m)^n (A - \lambda_1 I)^{k_1-1} \vv_1 = \vec{0}.
        \end{equation}
        Da $(A - \lambda_1 I)^{k_1-1} \vv_1 \not= \vec{0}$ und alle Eigenwerte verschieden sind, folgt $\alpha_1 = 0$.
        
        Für $\alpha_2, \ldots, \alpha_m$ können wir analog vorgehen, also gilt $\alpha_1 = \cdots = \alpha_m = 0$.
    \end{proof}
    
    \subsection{Proposition~\ref{prop:hau_eig}}
    
    \begin{proof}
        Sei $\lambda' \in \sigma(B)$ (Satz~\ref{thm:ev_exist}) und $\wv \in \Eig_B(\lambda')$ mit $\wv \not= \vec{0}$.
        
        Zunächst gilt
        \begin{equation}
            AV\wv = VB\wv = \lambda' V\wv
        \end{equation}
        und somit
        \begin{equation}
            (A - \lambda I) V\wv = (\lambda' - \lambda) V\wv.
        \end{equation}
        Da $V\wv \in \Hau_A(\lambda)$, erhalten wir damit
        \begin{equation}
            (\lambda' - \lambda)^n V\wv = (A - \lambda I)^n V\wv = \vec{0}.
        \end{equation}
        Da $\wv \not= \vec{0}$ und somit auch $V\wv \not= \vec{0}$, folgt $(\lambda' - \lambda)^n = 0$ und schließlich $\lambda' = \lambda$.
    \end{proof}
    
    \subsection{Proposition~\ref{prop:jord_chain_indep}}
    
    \begin{proof}
        Seien $\vv_1, \ldots, \vv_\ell \in \Hau_A(\lambda)$ eine Jordan-Kette zu~$\lambda$, wobei $\ell \leq L$ die Kettenlänge bezeichnet, und $\alpha_1, \ldots, \alpha_\ell \in \C$, sodass
        \begin{equation}
            \alpha_1 \vv_1 + \cdots + \alpha_\ell \vv_\ell = \vec{0}.
        \end{equation}
        
        Per Konstruktion gilt dann
        \begin{equation}
            \alpha_1 (A - \lambda I)^{\ell-1} \vv_\ell + \cdots + \alpha_{\ell-1} (A - \lambda I) \vv_\ell + \alpha_\ell \vv_\ell = \vec{0}.
        \end{equation}
        Eine Multiplikation dieser Gleichung mit $(A - \lambda I)^{\ell-1}$ ergibt
        \begin{alignat}{3}
                                 &  & \alpha_1 (A - \lambda I)^{2(\ell-1)} \vv_\ell + \cdots + \alpha_{\ell-1} (A - \lambda I)^\ell \vv_\ell + \alpha_\ell (A - \lambda I)^{\ell-1} \vv_\ell & = \vec{0}   \\
            \Leftrightarrow\quad &  & \alpha_\ell (A - \lambda I)^{\ell-1} \vv_\ell                                                                                                          & = \vec{0}.
        \end{alignat}
        Da $(A - \lambda I)^{\ell-1} \vv_\ell \not= \vec{0}$, folgt $\alpha_\ell = 0$.
        
        Für $\alpha_{\ell-1}, \ldots, \alpha_1$ können wir analog vorgehen, also gilt $\alpha_1 = \cdots = \alpha_\ell = 0$.
    \end{proof}
    
    \subsection{Proposition~\ref{prop:jord_chain_inv}}
    
    \begin{proof}
        Sei $\vv_1, \ldots, \vv_\ell$ eine Jordan-Kette zu~$\lambda$, d.\,h. insbesondere
        \begin{equation}
            A\vv_1 = \lambda \vv_1
            \quad \text{und} \quad
            A\vv_k = \vv_{k-1} + \lambda \vv_k \quad \text{für $k = 2, \ldots, \ell$},
        \end{equation}
        und
        \begin{equation}
            \vv \defas \alpha_1 \vv_1 + \cdots + \alpha_\ell \vv_\ell \quad \text{für $\alpha_1, \ldots, \alpha_\ell \in \C$}
        \end{equation}
        eine Linearkombination von $\vv_1, \ldots, \vv_\ell$.
        Dann gilt
        \begin{align}
            A\vv & = \alpha_1 A\vv_1 + \cdots + \alpha_\ell A\vv_\ell                                                                   \\
                 & = \alpha_1 \lambda \vv_1 + \alpha_2 (\vv_1 + \lambda \vv_2) + \cdots + \alpha_\ell (\vv_{\ell-1} + \lambda \vv_\ell) \\
                 & = \tilde{\alpha}_1 \vv_1 + \cdots + \tilde{\alpha}_\ell \vv_\ell
        \end{align}
        für Koeffizienten $\tilde{\alpha}_1, \ldots, \tilde{\alpha}_\ell \in \C$, wobei
        \begin{equation}
            \tilde{\alpha}_k \defas \alpha_k \lambda + \alpha_{k+1} \quad \text{für $k = 1, \ldots, \ell-1$}
            \quad \text{und} \quad
            \tilde{\alpha}_\ell \defas \alpha_\ell \lambda.
        \end{equation}
        Also ist~$A\vv$ wieder eine Linearkombination von $\vv_1, \ldots, \vv_\ell$.
    \end{proof}
    
\end{appendices}

\section*{Danksagung}

Mein herzlicher Dank geht an die beiden Korrekturleser Benedikt Zönnchen (Hochschule München) und Nicolai Palm (LMU München), die mit ihrem wachen Geist und kritischen Auge zum Feinschliff dieses Artikels beigetragen haben.

Zu guter Letzt möchte ich mich noch bei meinen zwischenzeitlichen Gastgebern Herr und Frau Bürkl dafür bedanken, dass sie mir bei interessanten und interessierten Abendgesprächen die Gelegenheit gegeben haben, meiner Hingabe an die Schönheit und Eleganz mathematischer Strukturen hin und wieder Ausdruck zu verleihen.


\begin{thebibliography}{1}

\bibitem{axler1995down}
S.~Axler.
\newblock Down {W}ith {D}eterminants!.
\newblock {\em Am. Math. Mon.}, 102(2):139--154, 1995.

\bibitem{axler2024linear}
S.~Axler.
\newblock {\em Linear {A}lgebra {D}one {R}ight}.
\newblock Springer Nature, Cham, 2024.

\end{thebibliography}
\end{document}